\documentclass[12pt]{article}
\usepackage[english]{babel}
\usepackage{latexsym,amssymb,amsmath,amsthm,amsfonts}
\usepackage{graphicx}

\def\'#1{\ifx#1i{\accent"13 \i}\else{\accent"13 #1}\fi}

\newtheorem{theorem}{Theorem}[section]

\title{\sc The 6-girth-thickness of the complete graph}
\author{\sc H{\' e}ctor Casta{\~ n}eda-L{\' o}pez \footnotemark[2]
\and \sc Pablo C. Palomino \footnotemark[3]
\and \sc Andrea B. Ramos-Tort \footnotemark[3]
\and \sc Christian Rubio-Montiel \footnotemark[4]
\and \sc Claudia Silva-Ruiz \footnotemark[3]
}

\begin{document}
\maketitle
\def\thefootnote{\fnsymbol{footnote}}
\footnotetext[2]{Facultad de Ciencias, Universidad Aut{\' o}noma del Estado de M{\' e}xico, 50000, Toluca, Mexico, {\tt hcastanedal854@alumno.uaemex.mx}.}
\footnotetext[3]{Facultad de Ciencias, Universidad Nacional Aut{\' o}noma de M{\' e}xico, 04510, Mexico City, Mexico, {\tt [pablop96|ramos{\_}tort|callame]@ciencias.unam.mx}.}
\footnotetext[4]{Divisi{\' o}n de Matem{\' a}ticas e Ingenier{\' i}a, FES Acatl{\' a}n, Universidad Nacional Aut{\' o}noma de M{\' e}xico, 53150, Naucalpan, Mexico, {\tt christian@apolo.acatlan.unam.mx}.}

\renewcommand{\thefootnote}{\arabic{footnote}}

\begin{abstract}

\noindent\rule{14.5cm}{0.4pt}

The $g$-girth-thickness $\theta(g,G)$ of a graph $G$ is the minimum number of planar subgraphs of girth at least $g$ whose union is $G$. In this paper, we determine the $6$-girth-thickness $\theta(6,K_n)$ of the complete graph $K_n$ in almost all cases. And also, we calculate by computer the missing value of $\theta(4,K_n)$.

\noindent\rule{14.5cm}{0.4pt}

\end{abstract}
\textbf{Keywords:} Thickness, planar decomposition, complete graph, girth.

\textbf{2010 Mathematics Subject Classification:} 05C10.


\section{Introduction}
In this paper, all graphs are finite and simple. A graph in which any two vertices are adjacent is called a \emph{complete graph} and it is denoted by $K_n$ if it has $n$ vertices. If a graph can be drawn in the Euclidean plane such that no inner point of its edges is a vertex or lies on another edge, then the graph $G$ is called \emph{planar}. The \emph{girth} of a graph is the size of its shortest cycle or $\infty$ if it is acyclic. It is known that an acyclic graph of order $n$ has size at most $n-1$ and a planar graph of order $n$ and finite girth $g$ has size at most $\frac{g}{g-2}(n-2)$, see \cite{MR2368647}.

The \emph{thickness} $\theta(G)$ of a graph $G$ is the minimum number of planar subgraphs whose union is $G$. Equivalently, it is the minimum number of colors used in any edge coloring of $G$ such that each set of edges in the same chromatic class induces a planar subgraph.

The concept of the thickness was introduced by Tutte \cite{MR0157372}. The problem to determine the thickness of a graph $G$ is NP-hard \cite{MR684270}, and only a few of exact results are known, for instance, when $G$ is a complete graph \cite{MR0460162,MR0164339,MR0186573}, a complete multipartite graph \cite{MR0158388,MR3661075,rubio2019,MR3243852,MR3610769} or a hypercube \cite{MR0211901}.

Generalizations of the thickness for the complete graphs also have been studied such that the outerthickness $\theta_o$, defined similarly but with outerplanar instead of planar \cite{MR1100049}, and the $S$-thickness $\theta_S$, considering the thickness on a surface $S$ instead of the plane \cite{MR0245475}. 
The thickness has many applications, for example, in the design of circuits \cite{MR1079374}, in the Ringel's earth-moon problem \cite{MR1735339}, or to bound the achromatic numbers of planar graphs \cite{araujo2017complete}. See also \cite{MR1617664}.

In \cite{rubio20174}, the \emph{$g$-girth-thickness} $\theta(g,G)$ of a graph $G$ was defined as the minimum number of planar subgraphs of girth at least $g$ whose union is $G$. Indeed, the $g$-girth thickness generalizes the thickness when $g=3$ and the \emph{arboricity number} when $g=\infty$.

This paper is organized as follows. In Section \ref{Section2}, we obtain the $6$-girth-thickness $\theta(6,K_n)$ of the complete graph $K_n$ getting that $\theta(6,K_n)$ equals $\left\lceil \frac{n+2}{3}\right\rceil$, except for $n=3t+1$, $t\geq 4$ and $n\not=2$, for which $\theta(6,K_2)=1$. In Section \ref{Section3}, we show that there exists a set of 3 planar triangle-free subgraphs of $K_{10}$ whose union is $K_{10}$. The decomposition	was found by computer and, as a consequence, we disproved the conjecture that appears in \cite{rubio20174} about the missing case of the $4$-girth-thickness of the complete graph.


\section{Determining $\theta(6,K_n)$}\label{Section2}

A planar graph of $n$ vertices with girth at least $6$ has size at most $3(n-2)/2$ for $n\geq 6$ and size at most $n-1$ for $1\leq n \leq 5$, therefore, the $6$-girth-thickness $\theta(6,K_n)$ of the complete graph $K_n$ is at least \[\left\lceil \frac{n(n-1)}{3(n-2)}\right\rceil =\left\lceil \frac{n+1}{3}+\frac{2}{3n-6}\right\rceil =\left\lceil \frac{n+2}{3}\right\rceil\] for $n\geq 6$, as well as, $\left\lceil \frac{n+2}{3}\right\rceil$ for $n\in\{1,3,4,5\}$. We have the following theorem.

\begin{theorem} \label{teo1}
The $6$-girth-thickness $\theta(6,K_n)$ of $K_n$ is equal to $\left\lceil \frac{n+2}{3}\right\rceil$ except possibly when $n=3t+1$, for $t\geq4$, and $n\not=2$ for which $\theta(6,K_2)=1$.
\end{theorem}

\begin{proof}
To begin with, Figure \ref{Fig1} displays equality for $n=2,4,7,10$ with $\theta(6,K_n)=1,2,3,4$, respectively. The rest of the cases for $1\leq n\leq 10$ are obtained by the hereditary property of the induced subgraphs. We remark that the decomposition of $K_{10}$ was found by computer using the database of the connected planar graphs of order $10$ that appears in \cite{MR2973372}.

\begin{figure}[htbp]
\begin{center}	
\includegraphics{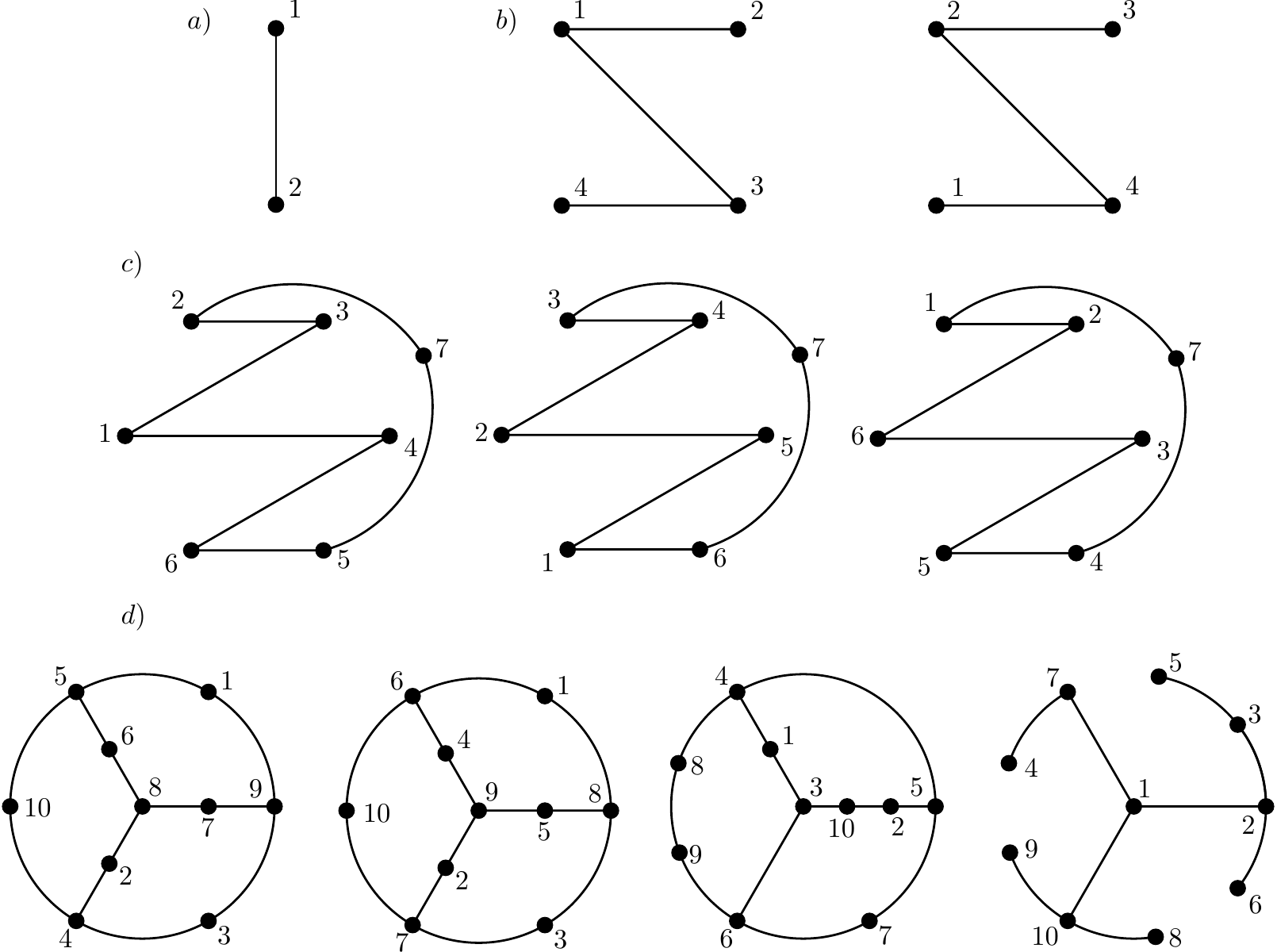}
\caption{\label{Fig1} A decomposition of $K_n$ into $\theta(6,K_n)$ planar subgraphs of girth at least $6$: $a)$ for $n=2$, $b)$ for $n=4$, $c)$ for $n=7$ and $d)$ for $n=10$.}
\end{center}
\end{figure}

Now, we need to distinguish two main cases, namely, when $t$ is even or $t$ is odd for $n=3t$, that is, when $n=6k$ and $n=6k+3$ for $k\geq2$. The cases $n=6k-1$ and $n=6k+2$, i.e., for $n=3t+1$, are obtained by the hereditary property of the induced subgraphs, that is, since $K_{6k-1}\subset K_{6k}$ and $K_{6k+2}\subset K_{6k+3}$, we have \[2k+1\leq \theta(6,K_{6k-1})\leq\theta(6,K_{6k})\text{ and}\]
\[2k+2\leq \theta(6,K_{6k+2})\leq\theta(6,K_{6k+3})\text{, respectively.}\]
Therefore, the case of $n=6k$ shows a decomposition of $K_{6k}$ into $2k+1$ planar subgraphs of girth at least $6$, while the case of $n=6k+3$ shows a decomposition of $K_{6k+3}$ into $2k+2$ planar subgraphs of girth at least $6$. Both constructions are based on the planar decomposition of $K_{6k}$ of Beineke and Harary \cite{MR0164339} (see also \cite{MR0460162,MR0186573,MR2626173}) but we use the combinatorial approach given in \cite{araujo2017complete}. Then, for the sake of completeness, we give a decomposition of $K_{6k}$ in order to obtain its usual thickness. In the remainder of this proof, all sums are taken modulo $2k$.

We recall that complete graphs of even order $2k$ are decomposable into a cyclic factorization of Hamiltonian paths, see \cite{MR2450569}. Let $G^x$ be a complete graph of order $2k$, label its vertex set $V(G^x)$ as $\{x_1,x_2,\dots,x_{2k}\}$ and let $\mathcal{F}^x_i$ be the Hamiltonian path with edges \[x_{i}x_{i+1},x_{i+1}x_{i-1},x_{i-1}x_{i+2},x_{i+2}x_{i-2},\dots,x_{i+k+1}x_{i+k},\] for all $i\in\{1,2,\dots,k\}$. The partition $\{E(\mathcal{F}^x_1),E(\mathcal{F}^x_2),\dots,E(\mathcal{F}^x_k)\}$ is such factorization of $G^x$. We remark that the center of $\mathcal{F}^x_i$ has the edge $e^x_i=x_{i+\left\lceil \frac{k}{2}\right\rceil}x_{i+\left\lceil \frac{3k}{2}\right\rceil}$, see Figure \ref{Fig2}. 
\begin{figure}[htbp]
\begin{center}
\includegraphics{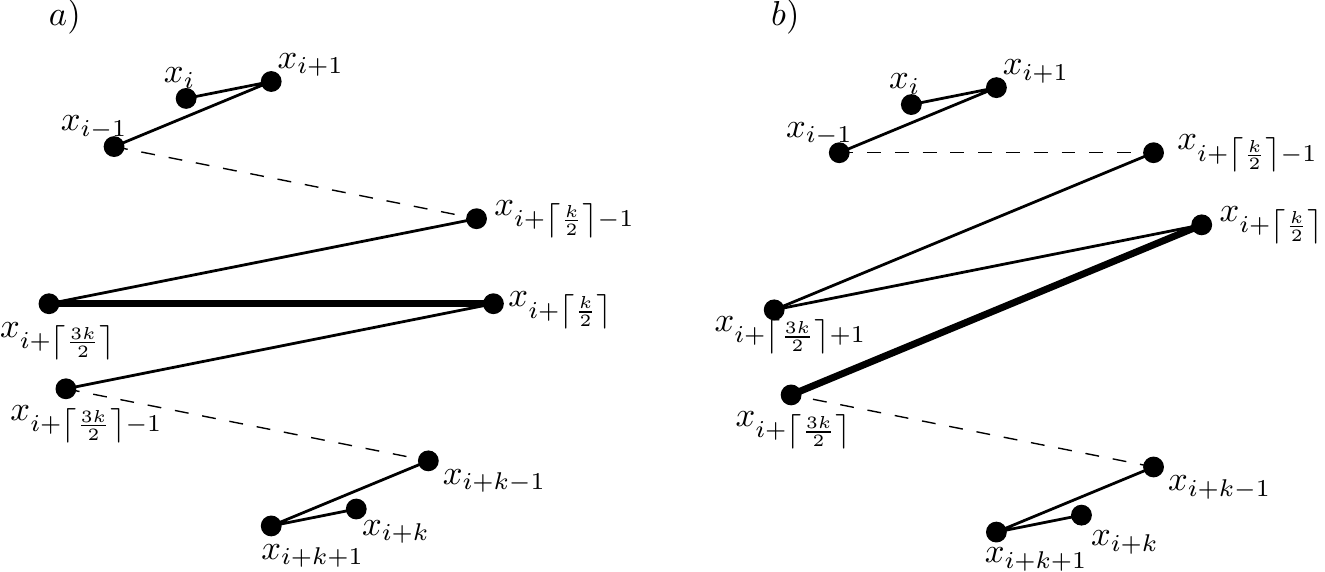}
\caption{The Hamiltonian path $\mathcal{F}^x_i$: Left $a)$ The edge $e^x_i$ in bold for $k$ odd. Right $b)$ The edge $e^x_i$ in bold for $k$ even.}\label{Fig2}
\end{center}
\end{figure}

Let $G^u$, $G^v$ and $G^w$ be the complete subgraphs of $K_{6k}$ having $2k$ vertices each of them and such that $G^w$ is $K_{6k}\setminus (V(G^u)\cup V(G^v))$. The vertices of $V(G^u)$, $V(G^v)$ and $V(G^w)$ are labeled as $\{u_1,u_2,\dots,u_{2k}\}$, $\{v_1,v_2,\dots,v_{2k}\}$ and $\{w_1,w_2,\dots,w_{2k}\}$, respectively.

Let $x$ be an element of $\{u,v,w\}$. Take the cyclic factorization $\{E(\mathcal{F}^x_1),E(\mathcal{F}^x_2),\dots,E(\mathcal{F}^x_k)\}$ of $G^x$ into Hamiltonian paths and denote as $P_{x_i}$ and $P_{x_{i+k}}$ the subpaths of $\mathcal{F}^x_i$ containing $k$ vertices and the leaves $x_i$ and $x_{i+k}$, respectively. We define the other leaves of $P_{x_i}$ and $P_{x_{i+k}}$ as $f(x_i)$ and $f(x_{i+k})$, respectively and according to the parity of $k$, that is (see Figure \ref{Fig2}), 
\begin{center}
$f(x_i) = \begin{cases} x_{i+\left\lceil \frac{3k}{2}\right\rceil} & \text{if  } k \text{ is odd, }\\ x_{i+\left\lceil \frac{k}{2}\right\rceil} & \text{if } k \text{ is even. } \end{cases}$ and $f(x_{i+k}) = \begin{cases} x_{i+\left\lceil \frac{k}{2}\right\rceil} & \text{if  } k \text{ is odd, }\\ x_{i+\left\lceil \frac{3k}{2}\right\rceil} & \text{if } k \text{ is even. } \end{cases}$
\end{center}
We remark that the set of edges $\{x_ix_{i+k}\colon 1\leq i\leq k\}$ is the same set of edges that $\{f(x_i)f(x_{i+k})\colon 1\leq i\leq k\}$.

Now, we construct the maximal planar subgraphs $G_1$, $G_2$,...,$G_{k}$ and a matching $G_{k+1}$ with $6k$ vertices each in the following way. Let $G_{k+1}$ be the perfect matching with the edges $u_ju_{j+k}$, $v_jv_{j+k}$ and $w_jw_{j+k}$ for $j\in\{1,2,\dots,k\}$.

For each $i\in\{1,2,\ldots,k\}$, let $G_i$ be the spanning planar graph of $K_{6k}$ whose adjacencies are given as follows:
we take the 6 paths, $P_{u_i},P_{u_{i+k}},P_{v_i},P_{v_{i+k}},P_{w_i}$ and $P_{w_{i+k}}$ and insert them in the octahedron with the vertices $u_{i},u_{i+k},v_{i},v_{i+k},w_{i}$ and $w_{i+k}$ as is shown in Figure \ref{Fig2} (Left). The vertex $x_j$ of each path $P_{x_j}$ is identified with the vertex $x_j$ in the corresponding triangle face and join all the other vertices of the path with both of the other vertices of the triangle face, see Figure \ref{Fig3} (Right).

\begin{figure}[htbp]
\begin{center}
\includegraphics{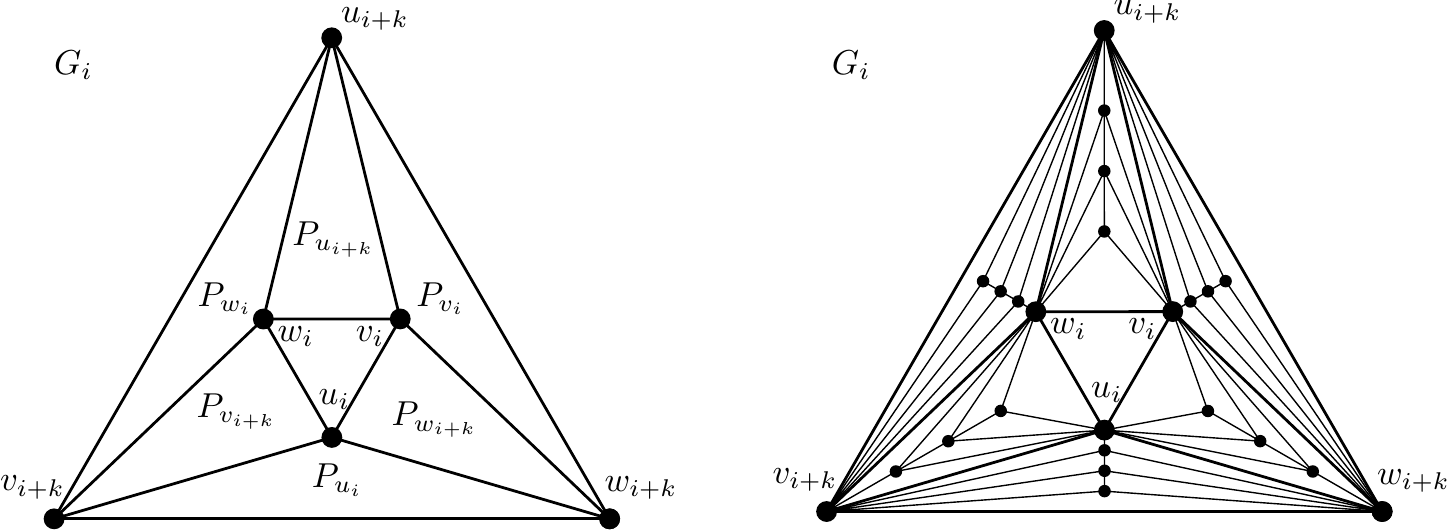}
\caption{(Left) The octahedron subgraph of the graph $G_{i}$. (Right) The graph $G_{i}$.}\label{Fig3}
\end{center}
\end{figure}

By construction of $G_i$, $K_{6k}=\overset{k+1}{\underset{i=1}{\bigcup}}G_{i}$, see \cite{MR0460162,MR0164339} to check a full proof. In consequence, the $k+1$ planar subgraphs $G_i$ show that $\theta(3,K_{6k})\leq k+1$ and then, $\theta(3,K_{6k})= k+1$ owing to the fact that $\theta(3,K_{6k})\geq\left\lceil \frac{\binom{6k}{2}}{3(6k-2)}\right\rceil =k+1.$

Now, we proceed to prove that $\theta(6,K_{6k})\leq 2k+1$ in Case 1 and $\theta(6,K_{6k+3})\leq 2k+2$ in Case 2. The main idea of both cases is divide each $G_i$ into two subgraphs of girth $6$ for any $i\in\{1,\dots,k\}$.

\begin{enumerate}
\item Case $n=6k$.

Consider the set of planar subgraphs $\{G_1,G_2,\dots,G_{k+1}\}$ of $K_{6k}$ which is described above.

\textbf{Step 1.} For each $i\in\{1,\dots,k\}$, remove the six edges of the triangles $u_iv_iw_i$ and $u_{i+k}v_{i+k}w_{i+k}$.

\textbf{Step 2.} For each $i\in\{1,\dots,k\}$, divide the obtained subgraph into two subgraphs $H_i^1$ and $H_i^2$ as follows: The maximum matching of $P_{x_i}$ incident to the vertex $f(x_i)$ belongs to $H_i^1$ (see dotted subgraph in Figure \ref{Fig4}) while the maximum matching of $P_{x_{i+k}}$ incident to the vertex $f(x_{i+k})$ belongs to $H_i^2$. 

Next, the rest of the edges joined to the vertices of the paths $P_{x_i}$ and $P_{x_{i+k}}$, in an alternative way from the exterior region to the region with the vertices $\{u_i,v_i,w_i\}$, belong to $H_i^1$ and $H_i^2$ respectively, such that the edges $f(w_i)u_{i+k}$, $f(v_i)w_{i+k}$ and $f(u_i)v_{i+k}$ belong to $H_i^1$ and the edges $f(w_i)v_{i+k}$, $f(v_i)u_{i+k}$ and $f(u_i)w_{i+k}$ belong to $H_i^2$, see Figure \ref{Fig4}.
\begin{figure}[htbp]
\begin{center}
\includegraphics{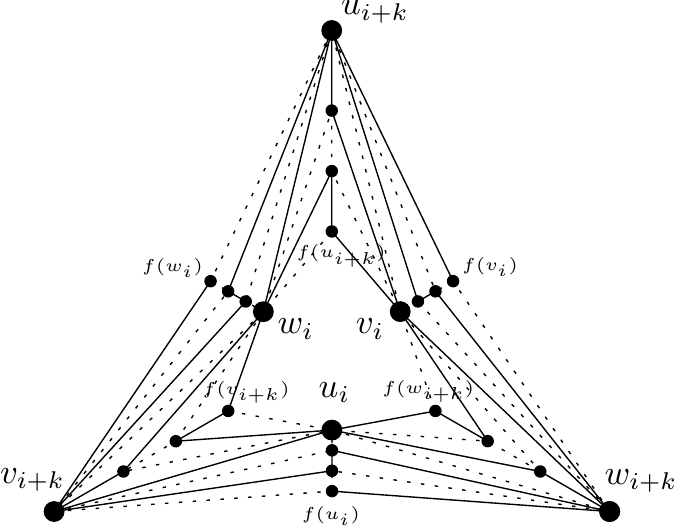}
\caption{\label{Fig4} Partial modification of the subgraph $G_i$.}
\end{center}
\end{figure}

\textbf{Step 3.} Consider the removed edges in Step 1, add the edges $f(v_{i+k})f(u_{i+k})$ and $f(u_{i+k})f(w_{i+k})$ to $H_i^1$ and the edges $f(w_{i})f(v_{i})$ and $f(v_{i})f(u_{i})$ to $H_i^2$, see Figure \ref{Fig5}. The rest of the edges removed in Step 1 are added to $G_{k+1}$ getting the subgraph $H_{k+1}$ which is the union of the paths $\{f(v_i),f(v_{i+k}),f(w_{i+k}),f(w_{i}),f(u_{i}),f(u_{i+k})\}$.	
\begin{figure}[htbp]
\begin{center}
\includegraphics{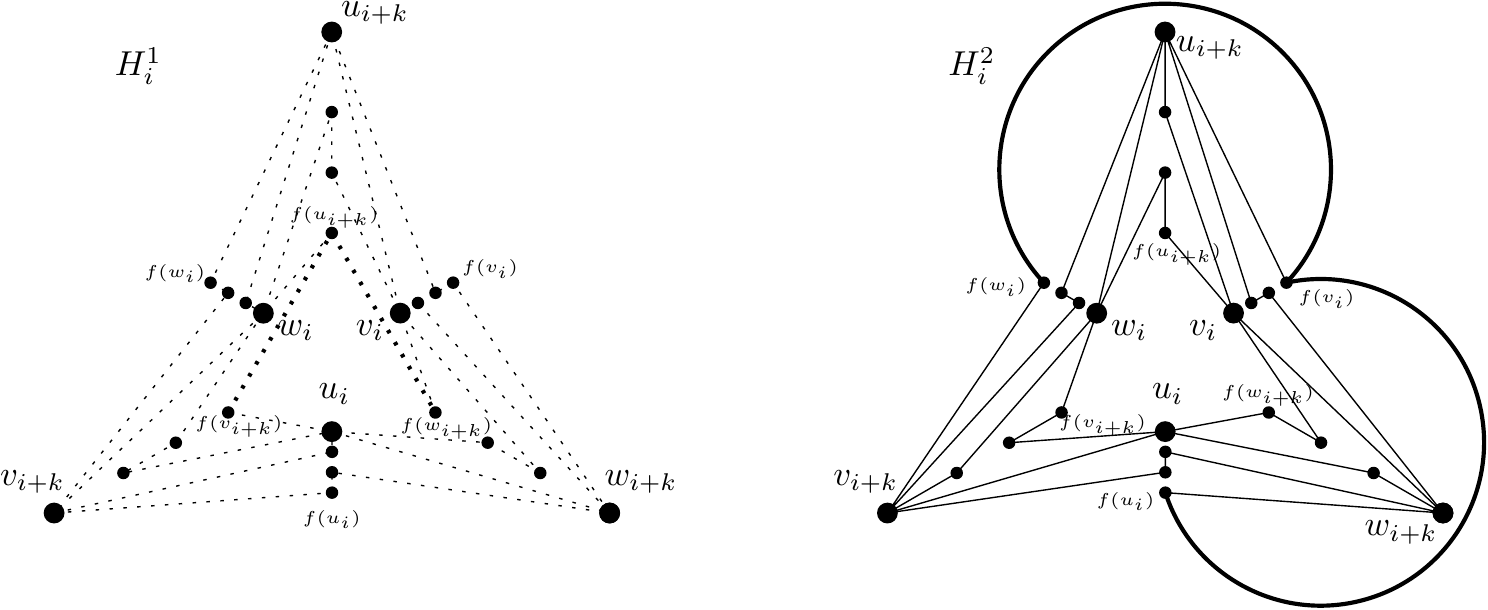}
\caption{\label{Fig5} Subgraphs $H_i^1$ and $H_i^2$ for the Case 1.}
\end{center}
\end{figure}

\item Case $n=6k+3$.

Consider the set of planar subgraphs $\{G_1,G_2,\dots,G_{k+1}\}$ of $K_{6k}$ which is described above as well as Step 1 and 2 of the previous case.

\textbf{Step 3.} Add three vertices $u$, $v$ and $w$ in the subgraphs $H_i^1$ and $H_i^2$, for each $i\in\{1,\dots,k\}$, and the edges $uw_i$, $uf(v_{i+k})$, $vu_i$, $vf(w_{i+k})$, $wv_i$, $wf(u_{i+k})$ into $H_i^1$ as well as the edges $uw_{i+k}$, $uf(v_{i})$, $vu_{i+k}$, $vf(w_{i})$, $wv_{i+k}$, $wf(u_{i})$ into $H_i^2$, see Figure \ref{Fig6}.
\begin{figure}[htbp]
\begin{center}
\includegraphics{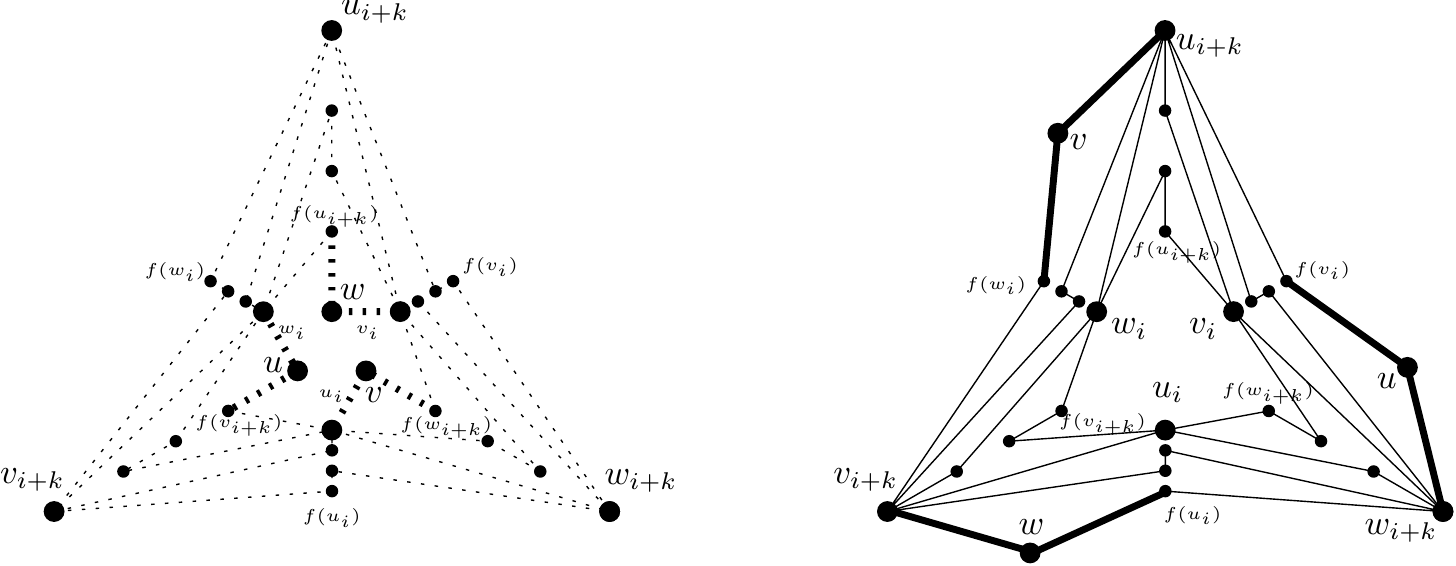}
\caption{\label{Fig6}Subgraphs $H_i^1$ and $H_i^2$ for Case 2.}
\end{center}
\end{figure}

\textbf{Step 4.} On one hand, remains to define the adjacencies between $u$, $v$, $w$ and all the adjacencies between $u$ and $u_i$, $v$ and $v_i$, $w$ and $w_i$, for each $j\in\{1,\dots,k\}$. On the other hand, the edges of the graph $G_{k+1}$ together with the removed edges of the Step 1 form a set of triangle prisms which we split into two subgraphs called $H^1_{k+1}$ and $H^2_{k+1}$ in the following way:

a) The adjacency $vw$ is in $H^1_{k+1}$ while the adjacencies $uv$ and $uw$ are in $H^2_{k+1}$, see Figure \ref{Fig7}.

b) The set of adjacencies $vv_{j+k}$, $ww_j$, $ww_{j+k}$ and $uu_{j+k}$ are in $H^1_{k+1}$ while the set of adjacencies $vv_{j}$, and $uu_{j}$ are in $H^2_{k+1}$, for each $j\in\{1,\dots,k\}$, see Figure \ref{Fig7}.

c) The subgraph $H^1_{k+1}$ contains the adjacencies $v_{j+k}v_j$, $v_ju_j$, $u_jw_j$ and $w_{j+k}u_{j+k}$ (a set of subgraphs $P_{4}\cup K_2$) and the subgraph $H^2_{k+1}$ contains the adjacencies $u_{j}u_{j+k}$, $u_{j+k}v_{j+k}$, $v_{j+k}w_{j+k}$,$w_{j+k}w_{j}$ and $w_jv_{j}$ (a set of subgraphs $P_6$) for all $j\in\{1,\dots,k\}$, see Figure \ref{Fig7}.
\begin{figure}[htbp]
\begin{center}
\includegraphics{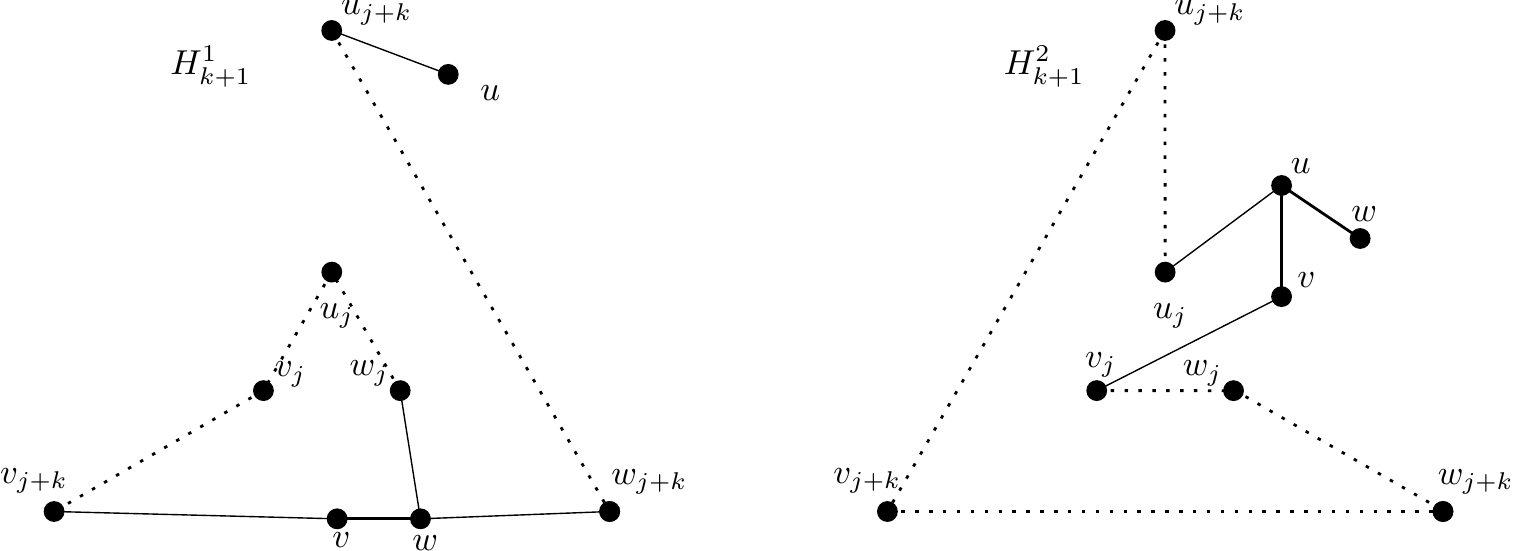}
\caption{\label{Fig7}Partial subgraphs $H^1_{k+1}$ and $H^2_{k+1}$.}
\end{center}
\end{figure}

\end{enumerate}
By the small cases and the two main cases, the theorem follows.
\end{proof}


\section{The $4$-girth thickness of $K_{10}$}\label{Section3}

In \cite{rubio20174}, Rubio-Montiel gave a decomposition of $K_n$ into $\theta(4,K_n)=\left\lceil \frac{n+2}{4}\right\rceil$ triangle-free planar subgraphs, except for $n=10$. In that case, it was bounded by $3\leq\theta(4,K_{10})\leq4$ and conjectured that the correct value was the upper bound. Using the database of the connected planar graphs of order $10$ that appears in  \cite{MR2973372} and the SageMath program, we found two decompositions of $K_{10}$ into $3$ planar subgraphs of girth at least $4$ illustrated in Figure \ref{Fig8}. In summary, the correct value of $\theta(4,K_n)$ was the lower bound and then, we have the following theorem.
\begin{theorem} \label{teo2}
The $4$-girth-thickness $\theta(4,K_n)$ of $K_n$ equals $\left\lceil \frac{n+2}{4}\right\rceil$ for $n\not=6$ and $\theta(4,K_6)=3$.
\end{theorem}
\begin{figure}[htbp]
\begin{center}
\includegraphics{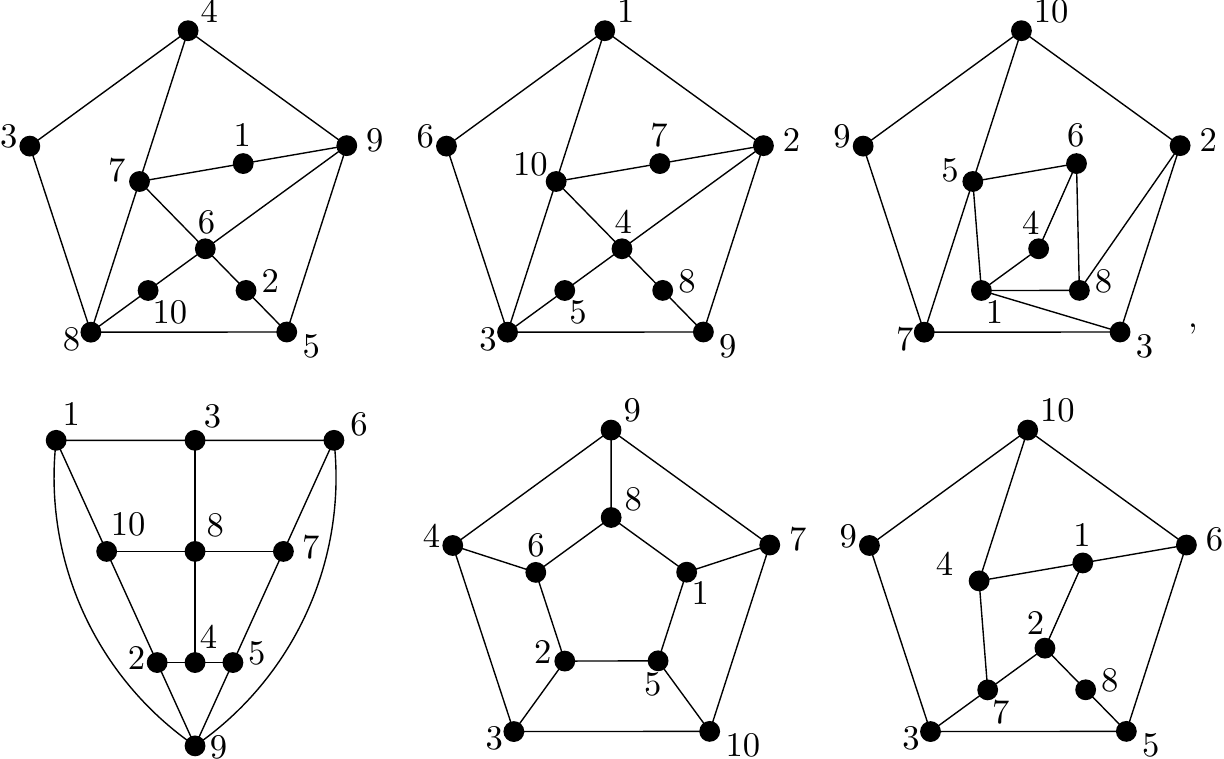}
\caption{\label{Fig8} Two planar decompositions of $K_{10}$ into three subgraphs of girth $4$.}
\end{center}
\end{figure}


\section*{Acknowledgments}
Part of the work was done during the IV Taller de Matem{\' a}ticas Discretas, held at Campus-Juriquilla, Universidad Nacional Aut{\' o}noma de M{\' e}xico, Quer{\' e}taro City, Mexico on June 11--16, 2017. We thank Miguel Raggi and Jessica S{\' a}nchez for their useful discussions and help with the SageMath program.

C. Rubio-Montiel was partially supported by PAIDI grant 007/19 and PAPIIT grant IN107218.


\end{document}